\documentclass[a4paper,11pt,english,reqno]{amsart}

\usepackage{babel,varioref}
\usepackage[ansinew]{inputenc}
\usepackage{amsmath}
\usepackage{amsfonts}
\usepackage{amscd}
\usepackage{enumerate}
\usepackage{graphics}
\usepackage{babel,varioref}
\usepackage{amscd}
\usepackage[all]{xy}
\usepackage[pdftex]{graphicx}
\DeclareGraphicsExtensions{.pdf,.png,.jpg}

\input amssym.def
\input amssym.tex

\def\reels{\mathbb{R}}
\def\complexes{\mathbb{C}}

\def\g {\mathfrak}

\DeclareMathOperator{\sign}{sign}\DeclareMathOperator{\Hom}{Hom}

\def\nac{\nabla}

\def\cg{g}
\def\kkg{\tilde g}
\def\fvf{v}
\def\rvf{\xi}
\def\parvf{r}
\def\cpvf{\tilde r}
\def\basic#1{{#1}^*}
\def\rec{\tilde{R}}
\def\curvd{R^{\mathcal{D}}}

\def\nad{\nabla^{\mathcal{D}}}

\newtheorem{theo}{Theorem}
\newtheorem{prop}[theo]{Proposition}

\newtheorem{lem}[theo]{Lemma}

\title{On the Holonomy of Kaluza-Klein metrics}

\author[Thomas Krantz]{Thomas Krantz}
\thanks{email: krantz@mathematik.hu-berlin.de\\mail: Institut für Mathematik, Humboldt-Universität, 10099 Berlin, Germany\\
The author acknowledges funding by the Fonds National de la Recherche (mesure d'accompagnement MA6, Research Ministry, G.-D. of Luxembourg) }

\makeindex 

\begin{document}
\begin{abstract}
We investigate Kaluza-Klein metrics with a recurrent light-like vector field over a pseudo-Riemannian manifold $(B,\cg)$.
\end{abstract}
\maketitle

\section{Introduction}

Principal $S^1$-bundles have been used by physicists to unify Einsteins general relativity theory with electrodynamics. This theory, by their main contributors known as Kaluza-Klein theory, has been later on generalized to bundles where the fiber is no longer an abelian Lie group(Yang-Mills theory). We refer to \cite{baum} for a general picture of the mathematical backgrounds of Gauge theory.

We focus here in particular on the construction of pseudo-Riemannian manifolds with the so called Kaluza-Klein-metric.
Given a pseudo-Riemannian manifold $(B,g)$, we consider a principal $S^1$-bundle $P$ over $B$ together with a connection form $A$ on it. For the Kaluza-Klein-metric a fundamental vector corresponding to the $S^1$-action has length $\pm 1$ whereas the horizontal spaces, corresponding to the connection form $A$, inherit the metric from the base manifold. Finally any horizontal vector is orthogonal to any vertical vector({\em i.e.} which is tangent to the fiber). The well known Boothby-Wang fibration now gives examples of such principal $S^1$-bundles over a projective symplectic manifold and it is known that if the base $B$ is a projective Kähler manifold, $P$ will carry a natural Sasakian structure.
We can fix transverse holonomy of the Sasakian manifold by fixing the holonomy of the base manifold.
We refer to \cite{blair} and \cite{bg2} for a detailed treatment of Sasakian geometry.

It is now known that Lorentzian connections on a $n+1$-dimensional manifold either admit full $\g s \g o(1,n)$-holonomy or there is a holonomy-invariant direction, case in which we speak of special Lorentzian holonomy.
By the work of L. Bérard Bergery, A. Ikemakhen (\cite{LBBAI}) and T. Leistner (\cite{leistner1}) indecomposable special holonomy
has been classified.
In section \ref{kkmrvf} we explore more generally ({\em i.e.} in the pseudo-Riemannian case) principal $S^1$-bundles with a Kaluza-Klein-metric admitting a holonomy-invariant direction or similarly a recurrent vector field. Under the condition that the recurrent vector field projects fiberwise on a fixed direction, it turns out that the recurrent vector field has to be parallel and projects onto a Killing field on the base manifold and that $P$ has to admit a flat connection. Inversely we can construct examples of pseudo-Riemannian manifolds with a recurrent vector field by considering $S^1$-bundles with a flat connection over a base manifold with a Killing field and form a Kaluza-Klein-metric. In particular we can consider base manifolds to be K-contact or Sasakian.

Finally we can iterate the construction: A first Boothby-Wang construction over a projective symplectic manifold yields a K-contact manifold, and a second $S^1$-bundle-construction yields a manifold on which there is a Kaluza-Klein-metric with a parallel vector field. In case the starting manifold was projective Kähler the special Lorentzian holonomy (of type II) is completely known.

{\sc Acknowledgements.} I thank Helga Baum, Mihaela Pilca, Thomas Neukirchner and Kordian Laerz among many others for useful discussions.

\section{Transverse connections}
Let $(B,\cg)$ be a pseudo-Riemannian manifold with Levi-Civita covariant derivative $\nac$.

Let $\xi$ be a vector field on $(B,\cg)$ such that $\cg(\xi,\xi)=\sigma$ with $\sigma=\pm 1$.

Note $\mathcal{D}=\mathcal{D}_\xi$ the sub-bundle $\xi^\perp$.
Note $p(X)$ the orthogonal projection of the vector field $X$ on $B$ onto $\mathcal{D}$.
Note $\nad$ the covariant derivative on the bundle $\mathcal{D}$ defined by:
$\nad_{\rvf} Y = p([\rvf,Y])$, $\nad_{X} Y=p(\nac_X Y)$ for $Y$ a section of $\mathcal{D}$ and $X$ orthogonal to $\rvf$.

It can be readily checked that $\nad$ is really a covariant derivative and that $\nad$ is torsion-free in the sense: $\nad_X Y - \nad_Y X = p([X,Y])$ for $X,Y$ sections of $\mathcal{D}$.
Furthermore $\nad$ preserves the bundle metric on $\mathcal{D}$ obtained by restriction of $\cg$.

Note $Hol^{\mathcal{D}}_o$ the holonomy group of the transverse connection in the point $o$ and $\g h \g o \g l^{\mathcal{D}}_o$ its Lie algebra.

\section{Kaluza-Klein metrics}
\subsection{Definitions and basic results}\label{KKC}
Given a pseudo-Riemannian manifold $(B,\cg)$, let $\nac$ be its Levi-Civita covariant derivative.
Let $\g g$ be the Lie algebra of the group $G=U(1)$. $S^1$ is the underlying manifold of the group $G$. $\g g$ identifies to $i \reels$.
We consider $(P,\pi,B)$ a principal $S^1$-bundle over the manifold $B$.
Let $A: TP\to \g g$ be a connection form on the bundle $P$, {i.e.}
for $g\in G$ and $X$ a vector field on $P$, $A(X \cdot g)=Ad(g^{-1})A(X)$ ($=A(X)$ in this case), and for $a\in \g g$, $\xi\in P$,
$A(\frac{d}{dt}_{t=0}(\xi\cdot \exp(ta)))=a$.

Let $\fvf$ be the vector field on $P$ such that $T\pi(\fvf)=0$ and $A(\fvf)$ is constant equal to $i$.

For $X$ a vector field on $B$, we call {\em horizontal lift of $X$} and note $\basic{X}$ the unique vector field on $P$ such that $A(\basic{X})=0$ and $T\pi(\basic{X})=X$.

Recall that the curvature form $\Omega$ in $\Gamma((T^*B\wedge T^*B)\otimes\g g)$ of the connection form $A$ verifies $\Omega(X,Y):=-A([\basic{X},\basic{Y}])$.
As $S^1$ is abelian, $\Omega$ is closed.
We have:

\begin{prop} \label{brackets}
$[\basic{X},\fvf]=0$,

$[\basic{X},\basic{Y}]=\basic{[X,Y]}+i\Omega(X,Y)\fvf$.
\end{prop}
\begin{proof}
Classic results.
\end{proof}

Fix $\sigma=\pm 1$. Note $\kkg$ the pseudo-Riemannian metric $\pi^*\cg + \sigma A\otimes A$ on $P$, called in the following {\em Kaluza-Klein-metric}. Let $D$ be the Levi-Civita covariant derivative corresponding to $\kkg$.

Note $\phi$ the section of $End_B(TB)$ such that $\cg(\phi X,Y)=\sigma \frac{i}{2}\Omega(X,Y)$.

\begin{prop}
$\sigma\frac{i}{2}(\nac \Omega)(X,Y,Z)=\cg((\nac_X \phi)Y,Z)$
\end{prop}

\begin{prop}
$D_{\fvf} \fvf=0$,

$D_{\fvf} \basic{X} = D_{\basic{X}}\fvf = \basic{(\phi X)}$

$D_{\basic{X}} \basic{Y} = \basic{(\nac_{X} Y)}+\frac{i}{2}\Omega(X,Y)\fvf$.
\end{prop}

\begin{proof}
Use the Koszul formula.
\end{proof}

\begin{lem}
$\fvf$ is a Killing vector field s.t. $\kkg(\fvf,\fvf)=-\sigma$.
\end{lem}

\begin{prop}\label{transverseHol}
Let $\mathcal{D}$ be the transverse bundle $\fvf^\perp$.
The transverse connection $D^{\mathcal{D}}$ verifies:
$D^{\mathcal{D}}_{\fvf} \basic{Y}=0$,
$D^{\mathcal{D}}_{\basic{X}}\basic{Y}=\basic{(\nac_{X} Y)}$.

The holonomy group in the point $o$ of the connection $D^{\mathcal{D}}$
verifies $Hol_o^{\mathcal{D}}=\pi^*(Hol_{\pi(o)}^\nac)$.
\end{prop}

\subsection{Boothby-Wang fibration}
Consider the $S^1$-bundle $P$ as in the preceding section.
It is known (see chapter 2 in \cite{blair}, resp. \cite{kob}) that $c_1(P)=[\frac{1}{2\pi i}\Omega]$ is an integral second De Rham cohomology class {\em i.e.} which lies in $H^2(B,\mathbb{Z})_b\subset H^2(B,\mathbb{R})$ ($b$ stands for Betti-part).

Inversely suppose that $\Omega$ is a closed $i\reels$-valued $2$-form such that $[\frac{1}{2\pi i}\Omega]\in H^2(B,\mathbb{Z})_b.$ As the principal $S^1$-bundles over $B$ are classified by
$H^2(B,\mathbb{Z})$, we can consider a bundle $P$ corresponding to $[\frac{1}{2\pi i}\Omega]$. If $B$ is simply connected, then the bundle $P$ is uniquely defined.
There is a (generally not unique) connection form $A$ on the bundle $P$ such that $\Omega$ is the curvature form of $A$.

\subsection{Examples of Kähler base manifolds}
If $(B,g,\Omega)$ is a Kähler manifold, the condition $[\frac{1}{2\pi i}\Omega]\in H^2(B,\mathbb{Z})_b$ means, by the Kodaira embedding theorem, that $B$ is projective, {\em i.e.} holomorphically embeds into some projective space $\complexes P^n$.
The construction of section \ref{KKC} gives then examples of Sasakian manifolds (see \ref{sasakian_examples}) with holonomy of the transverse bundle given by the holonomy of $(B,g)$: By a theorem of Y. Hatakeyama, the total space of the principal $S^1$-bundle is normal contact if and only if the base manifold is Kähler and projective(see \cite{hatakeyama}).

By the Berger theorem it is known that the list of holonomy algebras of irreducible, non locally-symmetric Kähler manifolds is restricted to $\g u(n)$, $\g s \g u(n)$ (Calabi-Yau), $\g s \g p(n)$ (hyperKähler).

\subsubsection{Projective Kähler symmetric spaces} Irreducible Kähler symmetric spaces are automatically Einstein. Furthermore if they are compact and simply-connected, they are projective(see \cite{besse}, 8.89, 8.2, 8.99). See also (\cite{besse}, 10.K) for the list of compact irreducible Kähler symmetric spaces.
\subsubsection{Projective base manifold with full $U(n)$-holonomy}
Consider $(X,g)$ a compact homogeneous Kähler manifold. If $X$ is simply-connected, it is known (see \cite{besse}, 8.2, 8.98-99) that $X$ admits a Kähler-Einstein metric $g$ with positive scalar curvature, and $X$ is projective.
If $(X,g)$ is not a symmetric space and irreducible, it admits necessarily then full $U(n)$-holonomy as Kähler-manifolds whose holonomy is included in $SU(n)$ are Ricci-flat.
\subsubsection{Projective Calabi-Yau base manifolds} It is known that manifolds with full $SU(n)$-holonomy are automatically projective as soon as their complex dimension is bigger than $3$.
\subsubsection{Projective hyperKähler base manifolds} Very little examples are known by today of projective hyperKähler manifolds. One series of such can be obtained in the following way: Take $X$ a projective K3-surface {\em e.g.} the Fermat quartic:
$$\{(x_0:x_1:x_2:x_3) | x_0^4+x_1^4+x_2^4+x_3^4=0 \}\subset \complexes P^3$$
$X$ itself and the Hilbert scheme $Hilb^n(X)$ of dimension $2n$ are then projective and hyperKähler.
See (\cite{ghj}, 21.1-2 and 26, and \cite{besse} 14.C) for more information about these examples and projective hyperKähler manifolds in general.

\section{Kaluza-Klein metrics with a recurrent vector field}\label{kkmrvf}

\subsection{Recurrent vector fields}

We say that the non-vanishing vector field $\rec$ on $P$ is {\em recurrent} if there is a scalar-valued $1$-form $\omega$ on $P$ such that $D \rec=\omega\otimes\rec$.

Note that if $\rec$ is a recurrent vector field, then $\langle\rec\rangle$ and $\rec^\perp$ are parallel distributions. In case the connection is torsion-free a parallel distribution is automatically integrable and tangent to a foliation of the manifold. In case $\rec$ is light-like, $\langle\rec\rangle$ is a sub-distribution of $\rec^\perp$ and the leaves corresponding to $\langle\rec\rangle$ are contained in the leaves corresponding to $\rec^\perp$.

\subsection{Screen bundle}
The quotient-bundle $\mathcal{S}:=\rec^\perp /\langle\rec\rangle$ is called the {\em screen bundle} corresponding to the recurrent vector field $\rec$. Note $q$ the corresponding canonical projection $\rec^\perp \to \mathcal{S}$. $\mathcal{S}$ carries a natural covariant derivative $D^{\mathcal{S}}$:
$$D^{\mathcal{S}}_X (qY):=q(D_X Y).$$
\begin{lem}
$D^{\mathcal{S}}$ is well defined.
\end{lem}

$D^{\mathcal{S}}$ restricts to any leaf $\mathcal{L}$ of the foliation $\mathcal{F}$ determined by $\rec^\perp$.

\subsection{Kaluza-Klein metrics with a recurrent vector field}

\begin{prop}\label{recurrentD}
\begin{enumerate}[(i)]
Given a pseudo-Riemannian manifold $(B,\cg)$, and let $\sigma$ be $\pm 1$.

\item Let $(P,\pi,B)$ a principal $S^1$-bundle over the manifold $B$ and connection form $A$.
Let $\kkg:=\pi^*\cg + \sigma A\otimes A$ be a metric on $P$.

Suppose there is a light-like recurrent vector field $\rec$ on $P$ of the form $f\basic{\rvf} + h \fvf$, where $\rvf$ is a vector field on $B$ such that $\cg(\rvf,\rvf)=\sigma$ and $f$ and $h$ are in $C^\infty(P)$.

Then $\varepsilon=\sign(fh)$ is locally constant $\pm 1$, and we have: $\rvf$ is a Killing vector field,
$D(\basic{\rvf}+\varepsilon \fvf)=0$, $\phi \rvf=0$, $\Omega(\rvf,\cdot)=0$, $\phi = -\varepsilon \nac \rvf$,

Furthermore $\Omega=-i\sigma\varepsilon d\eta$ for the $1$-form $\eta:=\cg(\rvf,\cdot)$. The connection form $A_0:=A+i\sigma\varepsilon\pi^*\eta$ on the bundle $\pi:P\to B$ is flat.

\item Reciprocally suppose $(B,\cg)$ admits a Killing vector field $\rvf$ such that $\cg(\rvf,\rvf)=\sigma$.
 Suppose the principal $S^1$-bundle $\pi:P\to B$ admits a flat connection $A_0$. Let $\kkg$ be the metric $\pi^*\cg + \sigma A\otimes A$ for the connection form $A:=A_0-i\sigma\varepsilon\pi^*\eta$ with $\eta:=\cg(\rvf,\cdot)$ and $\varepsilon=\pm 1$.
$P$ admits then the parallel light-like vector field $\basic{\rvf}+\varepsilon \fvf$ for the Levi-Civita connection associated to $\kkg$.

\end{enumerate}

\end{prop}

\begin{description}
\item[Notation]
We will note $\parvf$ the parallel vector field $\basic{\rvf}+\varepsilon \fvf$, and $\cpvf$ the complementary field $\basic{\rvf}-\varepsilon \fvf$.
\end{description}
\begin{proof}
\begin{enumerate}[1)]
\item The condition that $\rec$ is light-like writes $\kkg(\rec,\rec)=0$, giving $f^2\cdot \sigma-h^2 \cdot \sigma=0$, say $f=\varepsilon h$ with $\varepsilon=\pm 1$.
$$\rec=f\cdot(\basic{\rvf}+\varepsilon \fvf)$$
As a consequence $f$ is non vanishing.

\item $D \rec=\omega \otimes \rec$,

\begin{enumerate}[a)]
\item $D_{\fvf} \rec=\omega(\fvf)\rec$

$\begin{matrix}
\Rightarrow & \fvf(f)(\basic{\rvf}+\varepsilon \fvf)+f\cdot(D_{\fvf}\basic{\rvf}+\varepsilon D_{\fvf}\fvf) = \omega(\fvf)f\cdot(\basic{\rvf}+\varepsilon \fvf)\\
\Rightarrow & \fvf(f)(\basic{\rvf}+\varepsilon \fvf)+f\cdot(\phi \rvf) = \omega(\fvf)f\cdot(\basic{\rvf}+\varepsilon \fvf)\\
\end{matrix}$

From this follow the two equations:
\begin{eqnarray} \fvf(f) \fvf & = & \omega(\fvf) f \fvf\label{veq}\\ \fvf(f)\basic{\rvf}+f\cdot(\phi \rvf) & = & \omega(\fvf)f\basic{\rvf}\label{heq}
\end{eqnarray}

Equation~\ref{veq} gives \begin{equation}\omega(\fvf)=\fvf(\ln f)\label{lnf}\end{equation}
From equation~\ref{heq} follows \begin{equation}\phi \rvf=0\label{lnf2}\end{equation} otherwise stated \begin{equation}\Omega(\rvf,\cdot)=0\label{lnf2}\end{equation}

\item $D_{\basic{X}} \rec=\omega(\basic{X})\rec$
\begin{eqnarray*}\basic{X}(f)(\basic{\rvf}+\varepsilon \fvf)+f D_{\basic{X}}\basic{\rvf} + \varepsilon f D_{\basic{X}}\fvf & = & \omega(\basic{X})f \basic{\rvf}+\varepsilon \omega(\basic{X}) f \fvf\\
\end{eqnarray*}

From this follow the two equations:
\begin{eqnarray*}
\basic{X}(f)\fvf & = & \omega(\basic{X}) f \fvf\label{xlnf}\\
\basic{X}(f)\basic{\rvf}+f \cdot \basic{(\nac_{X}\rvf)}+\varepsilon f \cdot\basic{(\phi X)} & = & \omega(\basic{X})f \basic{\rvf}\label{xcurv}
\end{eqnarray*}

then
\begin{eqnarray}
\basic{X}(f) & = & \omega(\basic{X}) f \label{xlnf2}\\
\basic{(\nac_{X}\rvf)} & = & -\varepsilon \basic{(\phi X)}\label{xcurv2}
\end{eqnarray}

The two equations~\ref{lnf} and \ref{xlnf2} give
\begin{equation}
\omega = d(\ln f) \label{dlnf}\end{equation}
From equation \ref{dlnf} follows in particular that the vector field $\rec_0 := \basic{\rvf}+\varepsilon \fvf$ is parallel.

Equation \ref{xcurv2} rewrites:
\begin{equation}
\phi X = -\varepsilon \nac_{X}\rvf\label{xcurv3}
\end{equation}
Using the definition of $\phi$ and the antisymmetry of $\Omega$, one sees from this equation that $\rvf$ is a Killing vector field.

The relation $\Omega=-i\varepsilon \sigma d\eta$ follows from the fact: $d\eta(X,Y)=X(\eta(Y))-Y(\eta(X))-\eta([X,Y])=2\cg(-\nac_X \rvf,Y)$, true if $\rvf$ is a Killing vector field, and consequently $d\eta(X,Y)=2\varepsilon \cg(\phi X, Y)=\sigma \varepsilon i \Omega(X,Y)$
\end{enumerate}
\end{enumerate}
The reverse implication is immediate after the preceding discussion.
\end{proof}

Note that under the conditions of the proposition $\parvf^\perp$ is exactly the horizontal distribution of the connection form $A_0$.

\begin{prop}
Under the conditions of proposition \ref{recurrentD}, we have:
$\mathcal{L}_{\rvf} \eta=0$, $\mathcal{L}_{\rvf} d \eta=0$,  $\mathcal{L}_{\rvf} \phi=0$, $\nac_{\rvf}\phi=0$.
\end{prop}

\begin{proof}
Since $\mathcal{L}_{\rvf}=d\circ\iota_{\rvf} + \iota_{\rvf}\circ d$, the two first statements follow from $\eta(\rvf)=\sigma$ and $d\eta(\rvf,\cdot)=0$.

For the third note
$$0=-\frac{\sigma \varepsilon}{2}(\mathcal{L}_{\rvf} d \eta)(X,Y)=\rvf \cg(\phi X,Y)-\cg(\phi [\rvf,X], Y) - \cg(\phi X, [\rvf,Y])$$
$$=(\mathcal{L}_{\rvf} \cg)(\phi X,Y)+\cg((\mathcal{L}_{\rvf} \phi)X,Y).$$
From this follows $\mathcal{L}_{\rvf} \phi=0$.

In particular it implies: $\nac_{\rvf}(\phi X)- \nac_{\phi X} \rvf = \phi(\nac_{\rvf} X - \nac_{X} \rvf)$ and so $(\nac_{\rvf}\phi)X = \nac_{\phi X}\rvf-\phi(\nac_X \rvf)=-\varepsilon \phi^2 X + \varepsilon \phi^2 X = 0$, proving the last statement.
\end{proof}

\begin{lem}
Under the conditions of proposition \ref{recurrentD}, we have:
$(\mathcal{L}_{\phi X} \eta)(Y)=d\eta(\phi X,Y)$.
\end{lem}
\begin{proof}
This follows from the definition of $\mathcal{L}$ and from $\eta(\phi X)=0$.
\end{proof}

\subsection{Transverse connection on $\rvf^\perp$ in case $D$ admits a recurrent vector field}

Note $\mathcal{D}=\mathcal{D}_{\rvf}$ the sub-bundle $\rvf^\perp$ and let $\nad$ be the transverse connection corresponding to $\rvf$ on the bundle $\mathcal{D}$.

\begin{prop} Let $D$ be as in proposition \ref{recurrentD}.
Let $X$ and $Y$ be sections of $\mathcal{D}$.
\begin{eqnarray*}
D \parvf& = & 0,\\
D_{\parvf}\basic{Y} & = & \basic{(\nad_{\rvf}Y)},\\
D_{\parvf}\cpvf & = & 0,\\
D_{\basic{X}}\basic{Y} & = & \basic{(\nad_{X}Y)}+\varepsilon\frac{i}{2}\Omega(X,Y)\parvf,\\
D_{\basic{X}}\cpvf & = & - 2 \varepsilon \basic{(\phi X)},\\
D_{\fvf}\basic{Y} & = & \basic{(\phi Y)},\\
D_{\fvf}\cpvf & = & 0.\\
\end{eqnarray*}
\end{prop}

\begin{prop}\label{curvature} Let $D$ be as in proposition \ref{recurrentD}. Let $X, Y, Z$ be sections of $\mathcal{D}$.
The curvature of the covariant derivative $D$ expresses by:
\begin{eqnarray*}
R(\cdot,\cdot)\parvf& = & 0,\\
R(\parvf,\basic{X})\basic{Y} & = & \basic{(\curvd(\rvf,X)Y)},\\
R(\parvf,\basic{X})\cpvf& = & 0,\\
R(\basic{X},\basic{Y})\basic{Z} & = & \basic{(\curvd(X,Y)Z)}+\sigma \varepsilon\cg((\nad_X \phi)Y-(\nad_Y \phi)X,Z)\parvf,\\
R(\basic{X},\basic{Y})\cpvf& = & -2\varepsilon\basic{((\nad_X\phi)Y-(\nad_Y\phi)X)},\\
R(\parvf,\fvf)& = & 0,\\
R(\basic{X},\fvf)\basic{Y} & = & \basic{((\nad_X\phi)Y)}+\varepsilon\sigma\cg(\phi X,\phi Y)\parvf,\\
R(\basic{X},\fvf)\cpvf & = & 2 \varepsilon\basic{(\phi^2 X)}.\\
\end{eqnarray*}
\end{prop}

Let $\mathcal{S}:=\parvf^\perp /\langle\parvf\rangle$ be the screen bundle corresponding to $\parvf$. Note $q$ the corresponding canonical projection.

As we saw before the covariant derivative $D^{\mathcal{S}}$ restricts to any leaf $\mathcal{L}$ of the foliation $\mathcal{F}$ associated to $\mathcal{D}$. Note $\g h \g o \g l^{\mathcal{L},\mathcal{S}}_o$ its holonomy algebra.

Call $\iota$ the $C^\infty(B)$-linear mapping $\Gamma(TB) \to \Gamma(TP)$ defined by $\iota(\rvf)=\parvf$, $\iota(X)=\basic{X}$, for $X$ a section of $\mathcal{D}$. Note that $\iota$ is simply the horizontal lift associated to the connection form $A_0$.
The mapping $l^{\mathcal{S}}$ (or simply $l$) shall be defined as follows: For $p\in P$ and $v\in\mathcal{D}_{\pi(p)}$, let $l_p(v):= q(\iota(X))_p$ for $X$ a section of $\mathcal{D}$ such that $X_{\pi(p)}=v$. $l_p$ is an isomorphism between $\mathcal{D}_{\pi(p)}$ and $\mathcal{S}_{p}$.

\begin{prop}\label{restrictedRecurrentKKHolonomy}
For $X$ a vector field on $B$ and $Y$ a section of $\mathcal{D}$, $D^{\mathcal{S}}_{\iota X}q(\iota Y)=q(\iota(\nad_X Y))$.

For $o\in\mathcal{L}$, $l^{\mathcal{S}}_{o*}(\g h \g o \g l^{\mathcal{D}}_{\pi(o)})=\g h \g o \g l^{\mathcal{L},\mathcal{S}}_o$.
\end{prop}

%

\subsection{Examples of base manifolds}
\subsubsection{K-contact manifolds}
$(B^{2n+1},\phi,\rvf,\eta,g)$ is a K-contact manifold, meaning that:
\begin{enumerate}
\item $(B^{2n+1},\phi,\rvf,\eta,g)$ is an almost contact metric manifold, {\em i.e.}
$\phi$, $\rvf$, $\eta$ are respectively a $(1,1)$-tensor, a vector field and a $1$-form on the manifold $B$, $g$ is a Riemannian metric on $B$,
$\eta(\rvf)=1$, $d\eta(\rvf,\cdot)=0$, $\phi^2=-I+\eta\otimes \rvf$,
$g(X,Y)=g(\phi X, \phi Y)+\eta(X)\eta(Y)$,
\item $\cg(\phi X,Y)=-d\eta(X,Y)$
\item $\rvf$ is a Killing vector field
\end{enumerate}

$(B^{2n+1},\phi,\rvf,\eta,g)$ verifies then the conditions of proposition \ref{recurrentD}(ii) with $\sigma=1$ and $\varepsilon=1$.

\subsubsection{Sasakian manifolds}\label{sasakian_examples}
$(B^{2n+1},\phi,\rvf,\eta,g)$ is a Sasakian manifold, meaning that:
\begin{enumerate}
\item $(B^{2n+1},\phi,\rvf,\eta,g)$ is a K-contact manifold
\item $\phi$ satisfies $(\nabla_X \phi)(Y)=g(X,Y)\rvf - \eta(Y)X$
\end{enumerate}

In this case we have $\nad \phi=0$, and as a consequence $Hol^{0,\mathcal{D}}_o \subset U(\mathcal{D}_o)$.
Additionally we see from proposition \ref{curvature} that:

\begin{prop}\label{curvatureSasaki}Suppose $(B^{2n+1},\phi,\rvf,\eta,g)$ is a Sasakian manifold. Let $X, Y, Z$ be sections of $\mathcal{D}=\rvf^\perp$.
The curvature of the covariant derivative $D$ on the manifold $P$ constructed as before expresses by:
\begin{eqnarray*}
R(\cdot,\cdot)\parvf & = & 0,\\
R(\parvf,\cdot) & = & 0,\\
R(\basic{X},\basic{Y})\basic{Z} & = & \basic{(\curvd(X,Y)Z)},\\
R(\basic{X},\basic{Y})\cpvf& = & 0,\\
R(\basic{X},\fvf)\basic{Y} & = & \cg(X,Y)\parvf,\\
R(\basic{X},\fvf)\cpvf & = & - 2 \basic{X}.\\
\end{eqnarray*}
\end{prop}

Note that $\phi$ is parallel along any path $\gamma$.
For $X,Y$ sections of $\mathcal{D}$:
$R(\basic{X},\fvf)\basic{Y} = \cg(X,Y)\parvf$,
as a consequence we obtain:

\begin{prop}
If $(B^{2n+1},\phi,\rvf,\eta,g)$ is a Sasakian manifold the holonomy algebra of the Levi-Civita covariant derivative $D$ of $g$ on the manifold $P$ is of the form:
For the decomposition $T_oP=\langle\parvf_o\rangle\oplus (\iota\mathcal{D})_o \oplus \langle \cpvf_o \rangle$.
$$\mathfrak{hol}_o^D=\left\{ \left( \begin{smallmatrix} 0 & u^* & 0\\0 & A & -u\\0 & 0 & 0\end{smallmatrix}
\right)|\; A\in \g g, u\in \Hom(\langle \cpvf_o \rangle,(\iota\mathcal{D})_o ) \right\},$$
where $\g g=\iota_{o*}\g h \g o \g l_{\pi(o)}^{\mathcal{D}}$ is a sub-Lie-algebra of $\mathfrak{u}((\iota\mathcal{D})_o )$.\end{prop}

Note that in the theorem appears $\g h \g o \g l_{\pi(o)}^{\mathcal{D}}$ and not $\g h \g o \g l_{\pi(o)}^{\mathcal{L},\mathcal{D}}$ like in proposition \ref{restrictedRecurrentKKHolonomy}. In the proof, the fact that $\phi$ commutes with parallel transport on the transverse bundle is crucial.

\section{Double bundles}
The base manifold can be obtained by a preliminary circle bundle construction. We discuss this possibility in the following:
Let $(B,\cg)$ be a pseudo-Riemannian manifold with Levi-Civita connection $\nac$. Let $\pi_1:P_1\to B$ be a first principal $S^1$-bundle over $B$ with connection form $A_1$ and metric $\kkg_1=\pi_1^*\cg+\sigma A_1 \otimes A_1$.

We have seen that the corresponding fundamental field $\rvf=\fvf$ is then a Killing field such that $\kkg_1(\fvf,\fvf)=-\sigma$.

Let $\mathcal{D}_1$ be the transverse bundle $\rvf^\perp$ equipped with the transverse covariant derivative $D^1$.
Let $\iota_p$ be the mapping $T_{\pi_1(p)}B\to \mathcal{D}_{1,p}$ defined by $\iota_p(X_p)=\basic{X}_p$ for $X$ a vector field on $B$.
$\iota_p$ is an isomorphism.

Consider a second principal $S^1$-bundle $\pi_2:P_2\to P_1$ admitting a flat connection form $A_{2,0}$.
Fix $\varepsilon_2:=\pm 1$.
Let $\eta_2:=\kkg_1(\rvf,\cdot)$ be a $1$-form on $P_1$.
Define the connection form $A_2$ by $A_2 := A_{2,0}+i\sigma\varepsilon_2 \pi_2^*\eta_2$
Let $\kkg_2$ be the metric $\pi_2^*\kkg_1 - \sigma A_2\otimes A_2$.
By proposition \ref{recurrentD} follows that
$P_2$ admits the parallel vector field $\parvf=\basic{\rvf}+\varepsilon_2 \fvf_2$ for the Levi-Civita connection associated to $\kkg_2$.

Let $\mathcal{L}$ be a leaf of the foliation determined by $\mathcal{D}_2:=\parvf^\perp$ and $o$ be a point contained in it.
Let $\mathcal{S}$ be the screen bundle determined by $\parvf$.

\begin{prop}
$\g h \g o \g l^{\mathcal{L},\mathcal{S}}_o$ is isomorphic to $\g h \g o \g l^{\nac}_{\pi_1\pi_2(o)}$
\end{prop}
\begin{proof}
By propositions \ref{transverseHol} and \ref{restrictedRecurrentKKHolonomy} the holonomy algebra $\g h \g o \g l^{\mathcal{L},\mathcal{S}}_o$ verifies:
$\g h \g o \g l^{\mathcal{L},\mathcal{S}}_o=l^{\mathcal{S}}_{o*}(\iota_{\pi_2(o)*}(\g h \g o \g l^{\nac}_{\pi_1\pi_2(o)}))$
\end{proof}

\end{document}